\documentclass{article}

\usepackage{amsmath,amssymb,amsthm,mathrsfs}
\usepackage{graphicx}

\theoremstyle{definition} \newtheorem*{definition}{Definition}
\theoremstyle{definition} \newtheorem{hypothesis}{Hypothesis}
\theoremstyle{definition} \newtheorem{remark}{Remark}
\theoremstyle{definition} \newtheorem{example}{Example}
\theoremstyle{theorem}    \newtheorem{theorem}{Theorem}
                          \newtheorem{corollary}[theorem]{Corollary}
                          \newtheorem{lemma}[theorem]{Lemma}

\begin{document}

\title{Semilinear Stochastic Evolution Equations with L\'evy Noise and Monotone Nonlinearity}
\author{Erfan Salavati, Bijan Z. Zangeneh\\ \\
Department of Mathematics, Sharif University of Technology\\
Tehran, Iran}
\date{}

\maketitle

\begin{abstract}
Semilinear stochastic evolution equations with multiplicative L\'evy noise and monotone nonlinear drift are considered. Unlike other similar work we do not impose coercivity conditions on coefficients. Existence and uniqueness of the mild solution is proved using an iterative method. The continuity of the solution with respect to initial conditions and coefficients is proved and a sufficient condition for exponential asymptotic stability of the solutions has been derived. The solutions are proved to have a Markov property. Examples on stochastic partial differential equations and stochastic delay equations are provided to demonstrate the theory developed. The main tool in our study is an It\^o type inequality which gives a pathwise bound for the norm of stochastic convolution integrals.
\end{abstract}

\vspace{2mm}
\vspace{2mm}


\section{Introduction}\label{section: introduction}

Stochastic evolution equations have been an active area of research for many years. There are two main approaches in the study of nonlinear stochastic evolution equations. First studies equations of type
\begin{equation*}
    dX_t=AX_t dt + f(X_t) dt + g(X_t)d W_t
\end{equation*}
in a Hilbert space where $A$ is the infinitesimal generator of a $C_0$ semigroup of linear operators, $W_t$ is a Wiener process or more generally a martingale and $f$ and $g$ are assumed to be Lipschitz. Among studies taking this approach one can note Da Prato and Zabczyk~\cite{DaPrato_Zabczyk_book}, in which the existence and uniqueness of the mild solution for stochastic evolution equations with Wiener noise is proved as well as Kotelenez~\cite{Kotelenez-1984} in which the general martingale noise is considered.

The second approach considers equations of type
    \[ dX_t = F(X_t)dt + G(X_t)dW_t \]
in a Hilbert space $H$ equipped with a Banach space $B$ with dense embeddings  $B \subset H \subset B^*$, where $W_t$ is a Wiener process with values in a Hilbert space and $F$ and $G$ are generally assumed to be unbounded nonlinear operators that satisfy certain monotonicity and coercivity properties. This approach is called the variational method. Among studies in this framework one can note Pardoux~\cite{Pardoux}, Krylov and Rozovskii~\cite{Krylov-Rozovskii}, Liu and R\"ockner~\cite{Liu-Rockner} and R\"ockner~\cite{Rockner}. As one typical example of SPDE with this approach, the stochastic porous media equation has been studied by Barbu, Da Prato and R\"ockner~\cite{Barbu-DaPrato-Rockner-1,Barbu-DaPrato-Rockner-2}.

Each of these two approaches are stochastic versions of well known deterministic methods in nonlinear analysis. Another deterministic method is the semigroup approach to semilinear evolution equations with monotone nonlinearities, and it first appeared in the works of Browder~\cite{Browder} and Kato~\cite{Kato}. This approach has been generalized to stochastic evolution equations in Zangeneh~\cite{Zangeneh-Thesis} and~\cite{Zangeneh-Paper} to study equations of the type
\begin{equation}\label{equation: wiener_noise}
    dX_t=AX_t dt + f(X_t) dt + g(X_t)d W_t,
\end{equation}
where $W_t$ is a Wiener process and $f$ has a monotonicity assumption, i.e. there exists a real constant $M$ such that $\langle f(x)-f(y) , x-y \rangle \le M \|x-y\|^2$. We call such $f$ a \emph{semimonotone} operator. In the case that $M=0$, $f$ is called a \emph{monotone} operator. Monotone operators are also called \emph{dissipative} operators in the literature and they are generalizations of decreasing real functions. Every operator of the form $f=g+h$ where $g$ is monotone and $h$ is Lipschitz, is a semimonotone operator and vice versa. Hence this approach is a generalization of the first approach. This generalization is useful since there are natural semimonotone functions which are not Lipschitz; examples include decreasing real functions, such as $-\sqrt[3]{x}$, or sum of a non differentiable decreasing function with a Lipschitz function. Figure~\ref{figure:semimonotone} shows a semimonotone real function.

\begin{figure}[ht] \label{figure:semimonotone}
		\centering
		\includegraphics[scale=.5]{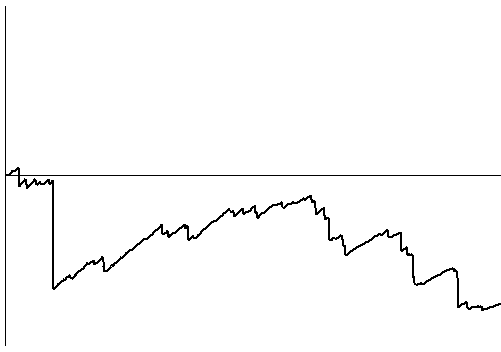}
		\caption{A semimonotone function}
\end{figure}

The semigroup approach to semilinear stochastic evolution equations with monotone nonlinearities also has an advantage to the variational method since it does not require the coercivity property. There are important examples, such as stochastic partial differential equations of hyperbolic type with monotone nonlinear terms, for which the generator does not satisfy the coercivity property and hence the variational method is not applicable directly to these equations. Pardoux~\cite{Pardoux} has developed a new theory to apply the variational method to the second order hyperbolic equations. But as is shown in Example~\ref{example: second_order_hyperbolic}, this problem can be treated directly in our setting. Another advantage of the semigroup approach to semilinear stochastic evolution equations with monotone nonlinearities is that it allows a unified treatment of different problems, such as stochastic partial differential equations of hyperbolic and parabolic type and stochastic delay differential equations.

There are a number of papers that have considered monotone (dissipative) nonlinearities but they assume the special case of additive white noise, see e.g Da Prato and R\"ockner~\cite{Da Prato-Rockner-dissipative, Da Prato-Rockner-singular}. The general case of multiplicative white noise has been considered by Zangeneh~\cite{Zangeneh-Paper} and existence and uniqueness of the mild solution has been proved.

Jahanipur and Zangeneh~\cite{Jahanipour-Zangeneh} has derived sufficient conditions for exponential asymptotic stability of solutions of~\eqref{equation: wiener_noise}. Jahanipur~\cite{Jahanipour-delay} considered stochastic delay evolution equations and proved existence and stability of mild solutions. Jahanipur~\cite{Jahanipur-functional-stability} generalized the results to stochastic functional evolution equations with coefficients depending on the past path of the solution. Hamedani and Zangeneh~\cite{Hamedani-Zangeneh-existence} considered a stopped version of~\eqref{equation: wiener_noise} and proved existence and uniqueness of the solution using a stopped maximal inequality for $p$-th moment of stochastic convolution integrals, which they proved in~\cite{Hamedani-Zangeneh-stopped}. Dadashi and Zangeneh~\cite{Dadashi-Zangeneh} studied the large deviation principle for~\eqref{equation: wiener_noise}. Zamani and Zangeneh~\cite{Zamani-Zangeneh-random-motion} studied a limiting problem of such equations arising from random motion of highly elastic strings. Finally, Zangeneh~\cite{Zangeneh_Nualart} studied the stationarity of a mild solution to a stochastic evolution equation with monotone nonlinear drift.

In recent years some research has appeared on stochastic evolution equations with L\'evy noise which use the first and second approaches above, see e.g. Peszat and Zabczyk~\cite{Peszat-Zabczyk}, Albeverio, Mandrekar and R\"udiger~\cite{Albeverio-Mandrekar-Rudiger-2009} and Marinelli, Pr\'ev\^ot and R\"ockner~\cite{Marinelli-Prevot-Rockner} for the case of Lipschitz coefficients and Brze\'zniak, Liu and Zhu~\cite{Brzezniak-Liu-Zhu} for coercive and monotone coefficients. There are a number of works that have considered monotone (dissipative) coefficients but they assume the special case of additive L\'evy noise, see e.g Peszat and Zabczyk~\cite{Peszat-Zabczyk} and Albeverio, Mastrogiacomo and Smii~\cite{Albeverio-small-noise}. We should also mention the article by Marinelli and R\"ockner~\cite{Marinelli-Rockner} which considers dissipative nonlinear drift and multiplicative L\'evy noise but only the uniqueness of the mild solution is proved there. As far as we know, the existence and continuity of the mild solution with respect to initial conditions for equations with monotone nonlinear drift and multiplicative L\'evy noise has not been studied before.

In this article we are concerned with the equation
\[    dX_t=AX_t dt+f(t,X_t) dt + g(t,X_{t-})d W_t + \int_E k(t,\xi,X_{t-}) \tilde{N}(dt,d\xi),\]
where $W_t$ is a cylindrical Wiener process on a Hilbert space $K$ and $\tilde{N}(dt,d\xi)$ is a compensated Poisson random measure. We assume $f$ is semimonotone and $g$ and $k$ are Lipschitz. In section~\ref{section: Problem} the assumptions on coefficients are stated precisely. We prove the existence and uniqueness of the mild solution in Theorem~\ref{theorem:existence and uniqueness} in section~\ref{section: Existence and Uniqueness}. The continuity of the solution with respect to initial conditions and coefficients will be proved in Theorem~\ref{theorem: continuity I} in section~\ref{section: Continuity With Respect to Parameter}  and a sufficient condition for the exponential stability of the solutions will be derived in Corollary~\ref{theorem: Exponential Stability}. Section~\ref{section: Markov Property} is devoted to proving the Markov property of the solutions. Some of the statements have been published previously in~\cite{Proceedings}, but the proofs were just outlined.

Our method in proving the existence of the solution is a certain iterative method that is specific to this type of equations. The main tool in our study is an It\^o type inequality that gives a pathwise bound for the norm of stochastic convolution integrals, which has been proved in~\cite{Zangeneh-Paper} and will be stated in section~\ref{section: Stochastic Convolution Integrals}. Since the usual inequalities for stochastic convolution integrals are not applicable to our equation, the so called inequality plays a central role in our study.

In the last section we will provide some concrete examples that our results could apply. The examples consist of a stochastic delay differential equation and three semilinear stochastic partial differential equations.

We will use the notion of stochastic integration with respect to cylindrical Wiener process and compensated Poisson random measure.
For the definition and properties see~\cite{Peszat-Zabczyk} and~\cite{Albeverio-Mandrekar-Rudiger-2009}.

\section{The Problem}\label{section: Problem}

Let $H$ be a separable Hilbert space with inner product $\langle \, , \, \rangle$. Let $S_t$ be a $C_0$ semigroup on $H$ with infinitesimal generator $A:D(A)\to H$. Furthermore we assume the exponential growth condition on $S_t$, i.e. there exists a constant $\alpha$ such that $\| S_t \| \le e^{\alpha t}$. If $\alpha=0$, $S_t$ is called a contraction semigroup. We denote by $L_{HS}(K,H)$ the space of Hilbert-Schmidt mappings from a Hilbert space $K$ to $H$.

\begin{definition}
    $f:H\to H$ is called \emph{demicontinuous} if whenever $x_n \to x$, strongly in $H$ then $f(x_n)\rightharpoonup f(x)$ weakly in $H$.
\end{definition}

Let $(\Omega,\mathcal{F},\mathcal{F}_t,\mathbb{P})$ be a filtered probability space. Let $(E,\mathcal{E})$ be a measurable space and $N(dt,d\xi)$ a Poisson random measure on $\mathbb{R}^+ \times E$ with intensity measure $dt \nu(d\xi)$. Our goal is to study the following equation in $H$,
\begin{equation}\label{main_equation}
    dX_t=AX_t dt+f(t,X_t) dt + g(t,X_{t-})d W_t + \int_E k(t,\xi,X_{t-}) \tilde{N}(dt,d\xi),
\end{equation}
where $W_t$ is a cylindrical Wiener process on a Hilbert space $K$ and $\tilde{N}(dt,d\xi)=N(dt,d\xi)-dt\nu(d\xi)$ is the compensated Poisson random measure corresponding to $N$. We assume that $N$ and $W_t$ are independent. We also assume the following,

\begin{hypothesis}\label{main_hypothesis}
    \begin{description}

        \item[(a)] $f(t,x,\omega):\mathbb{R}^+\times H\times \Omega \to H$ is measurable, $\mathcal{F}_t$-adapted, demicontinuous with respect to $x$ and there exists a constant $M$ such that
            \[ \langle f(t,x,\omega)-f(t,y,\omega),x-y \rangle \le M \|x-y\|^2,\]

        \item[(b)] $g(t,x,\omega):\mathbb{R}^+\times H\times \Omega \to L_{HS}(K,H)$ and $k(t,\xi,x,\omega):\mathbb{R}^+\times E\times H\times \Omega \to H$ are predictable and there exists a constant $C$ such that
            \[ \| g(t,x,\omega)-g(t,y,\omega)\|_{L_{HS}(K,H)}^2 + \int_{E}\|k(t,\xi,x)-k(t,\xi,y)\|^2 \nu(d\xi) \le C \|x-y \|^2,\]

        \item[(c)] There exists a constant $D$ such that
            \[ \| f(t,x,\omega)\|^2 + \| g(t,x,\omega)\|_{L_{HS}(K,H)}^2 + \int_{E}\|k(t,\xi,x)\|^2 \nu(d\xi) \le D(1+\|x\|^2),\]
        \item[(d)] $X_0(\omega)$ is $\mathcal{F}_0$ measurable and square integrable.
    \end{description}

\end{hypothesis}

\begin{definition}
    By a \emph{mild solution} of equation~\eqref{main_equation} with initial condition $X_0$ we mean an adapted c\`adl\`ag process $X_t$ that satisfies
    \begin{multline}\label{mild_solution}
        X_t=S_t X_0+\int_0^t S_{t-s}f(s,X_s) ds+\int_0^t{S_{t-s}g(s,X_{s-})d W_s}\\
        + \int_0^t{\int_E {S_{t-s}k(s,\xi,X_{s-})} \tilde{N}(ds,d\xi).}
    \end{multline}
\end{definition}

\section{Stochastic Convolution Integrals}\label{section: Stochastic Convolution Integrals}

In this section we review some properties and results about convolution integrals of type $X_t=\int_0^t S_{t-s} dM_s$ where $M_t$ is a martingale. These are called stochastic convolution integrals. Kotelenez~\cite{Kotelenez-1982} has proved that stochastic convolution integrals always have a c\`adl\`ag version. Kotelenez~\cite{Kotelenez-1984} also gives a maximal inequality for stochastic convolution integrals.

\begin{theorem}[Kotelenez,~\cite{Kotelenez-1984}] \label{Theorem: Kotelenez inequality}
    There exists a constant $\mathbf{C}$ such that for any $H$-valued c\`adl\`ag locally square integrable martingale $M_t$ we have
    \[ \mathbb{E} \sup_{0\le t\le T} \|\int_0^t S_{t-s}dM_s\|^2 \le \mathbf{C} e^{4\alpha T} \mathbb{E}[M]_T.\]
\end{theorem}

\begin{remark}
    Hamedani and Zangeneh~\cite{Hamedani-Zangeneh-stopped} generalized this inequality to a stopped maximal inequality for $p$-th moment ($0<p<\infty$) of stochastic convolution integrals.
\end{remark}

Because of the presence of monotone nonlinearity in our equation, we need an energy inequality for stochastic convolution integrals. For this reason the following pathwise inequality for the norm of stochastic convolution integrals has been proved in Zangeneh~\cite{Zangeneh-Paper}.
\begin{theorem}[It\^o type inequality, Zangeneh~\cite{Zangeneh-Paper}]\label{theorem:ito type inequality}
    Let $Z_t$ be an $H$-valued c\`adl\`ag locally square integrable semimartingale. If
    \[ X_t=S_t X_0 + \int_0^t S_{t-s}dZ_s, \]
    then
    \begin{equation*}
        \lVert X_t \rVert ^2 \le e^{2\alpha t}\lVert X_0 \rVert ^2 + 2 \int_0^t {e^{2\alpha (t-s)}\langle X_{s-} , d Z_s \rangle}+\int_0^t {e^{2\alpha (t-s)}d[Z]_s},
    \end{equation*}
    where $[Z]_t$ is the quadratic variation process of $Z_t$.
\end{theorem}

\section{Existence and Uniqueness}\label{section: Existence and Uniqueness}

Our proof for the existence of a mild solution relies on an iterative method which in each step requires to solve a deterministic equation, i.e. an equation in which $\omega$ appears only as a parameter. The following theorem proved in Zangeneh~\cite{Zangeneh-measurability} and~\cite{Zangeneh-Thesis} guarantees the solvability of such equations and the measurability of the solution with respect to parameter.

Let $(\Omega,\mathcal{F},\mathcal{F}_t,\mathbb{P})$ be a filtered probability space and assume $f$ satisfies Hypothesis~\ref{main_hypothesis}-(a) and there exists a constant $D$ such that $\|f(t,x,\omega)\|^2 \le D (1+\|x\|^2)$ and assume $V(t,\omega)$ is an adapted process with c\`adl\`ag trajectories and $X_0(\omega)$ is $\mathcal{F}_0$ measurable.

\begin{theorem}[Zangeneh,~\cite{Zangeneh-measurability} and~\cite{Zangeneh-Thesis}]\label{theorem: measurability}
    With assumptions made above, the equation
        \[ X_t=S_t X_0 + \int_0^t S_{t-s}f(s,X_s,\omega) ds + V(t,\omega)\]
    has a unique measurable adapted c\`adl\`ag solution $X_t(\omega)$. Furtheremore
    \[	\|X(t)\| \le \|X_0\|+\|V(t)\|+\int_0^t e^{(\alpha+M)(t-s)} \|f(s,S_s X_0+V(s))\| ds, \]
\end{theorem}

\begin{remark}
    Note that the original theorem is stated for evolution operators and requires some additional assumptions, but those are automatically satisfied for $C_0$ semigroups. (See Curtain and Pritchard~\cite{Curtain-Pritchard} page 29, Theorem 2.21).
\end{remark}

\begin{theorem}[Existence and Uniqueness of the Mild Solution]\label{theorem:existence and uniqueness}
    Under the assumptions of Hypothesis~\ref{main_hypothesis}, equation~\eqref{main_equation} has a unique square integrable c\`adl\`ag mild solution with initial condition $X_0$.
\end{theorem}

    \begin{lemma}\label{lemma: alpha=0}
        It suffices to prove theorem~\ref{theorem:existence and uniqueness} for the case that $\alpha=0$.
    \end{lemma}
    \begin{proof} Define
        \begin{gather*}
           \tilde{S}_t= e^{-\alpha t} S_t ,\qquad \tilde{f}(t,x,\omega)=e^{-\alpha t}f(t,e^{\alpha t}x,\omega) ,\qquad \tilde{g}(t,x,\omega)=e^{-\alpha t}g(t,e^{\alpha t}x,\omega), \\
           \tilde{k}(t,\xi,x,\omega)=e^{-\alpha t}k(t,\xi,e^{\alpha t}x,\omega).
        \end{gather*}
        Note that $\tilde{S}_t$ is a contraction semigroup. It is easy to see that $X_t$ is a mild solution of equation~\eqref{main_equation} if and only if $\tilde{X}_t=e^{-\alpha t} X_t$ is a mild solution of equation with coefficients $\tilde{S},\tilde{f},\tilde{g},\tilde{k}$.
    \end{proof}

\begin{proof}[Proof of Theorem~\ref{theorem:existence and uniqueness}.]
    \emph{Uniqueness.}
    According to the lemma, we can assume $\alpha=0$. Assume that $X_t$ and $Y_t$ are two mild solutions with same initial conditions. Subtracting them we find
    \[ X_t-Y_t=\int_0^t S_{t-s} dZ_s,\]
    where
    \begin{multline*}
        dZ_t=(f(t,X_t)-f(t,Y_t))dt+(g(t,X_{t-})-g(t,Y_{t-}))dW_t\\
        +\int_E{(k(t,\xi,X_{t-})-k(t,\xi,Y_{t-}))d\tilde{N}}.
    \end{multline*}
    Applying It\^o type inequality (Theorem~\ref{theorem:ito type inequality}) for $\alpha=0$ to $X_t-Y_t$ we find
    \[ \| X_t-Y_t \| ^2 \le 2 \int_0^t {\langle X_{s-}-Y_{s-} , d Z_s \rangle}+[Z]_t. \]
    Taking expectations and noting that integrals with respect to cylindrical Wiener processes and compensated Poisson random measures are martingales, we find that
    \[ \mathbb{E}\lVert X_t-Y_t \rVert ^2 \le 2 \int_0^t {\mathbb{E}{\langle X_{s-}-Y_{s-} , f(s,X_s)-f(s,Y_s) \rangle} ds}+\mathbb{E}[Z]_t, \]
    where
    \[ \mathbb{E}[Z]_t = \int_0^t\mathbb{E}\|g(s,X_s)-g(s,Y_s)\|^2ds+\int_0^t\int_E{\mathbb{E}\|k(s,\xi,X_s)-k(s,\xi,Y_s)\|^2\nu(d\xi)ds}. \]
            Note that for a c\`adl\`ag function the set of discontinuity points is countable, hence when integrating with respect to Lebesgue measure, they can be neglected. Therefore from now on, we neglect the left limits in integrals with respect to Lebesgue measure. Using assumptions of Hypothesis~\ref{main_hypothesis}-(a) and ~\ref{main_hypothesis}-(b) we find that
    \[ \mathbb{E}\lVert X_t-Y_t \rVert ^2 \le (2M+C) \int_0^t {\mathbb{E}\| X_s-Y_s\|^2 ds}.\]
    Using Gronwall's lemma we conclude that $X_t=Y_t$, almost surely.

    \emph{Existence.}
    It suffices to prove the existence of a solution on a finite interval $[0,T]$. Then one can show easily that these solutions are consistent and give a global solution. We define adapted c\`adl\`ag processes $X^{n}_t$ recursively as follows. Let $X^0_t=S_t X_0$. Assume $X^{n-1}_t$ is defined. Theorem~\ref{theorem: measurability} implies that there exists an adapted c\`adl\`ag solution $X^n_t$ of
    \begin{equation}\label{equation: proof of existence_iteration}
        X^n_t=S_t X_0 + \int_0^t S_{t-s}f(s,X^n_s) ds + V^n_t,
    \end{equation}
    where
        \[ V^{n}_t= \int_0^t{S_{t-s}g(s,X^{n-1}_{s-})d W_s} + \int_0^t{\int_E {S_{t-s}k(s,\xi,X^{n-1}_{s-})} \tilde{N}(ds,d\xi)}. \]
    We wish to show that $\{X^n\}$ converges and the limit is the desired mild solution. This is done by the following lemmas.
    \begin{lemma}\label{lemma: finite_second_moment_iteration}
    	\[ \mathbb{E}\sup\limits_{0\le t\le T} \|X^n_t\|^2<\infty. \]
    \end{lemma}

	\begin{proof}
		We prove by induction on $n$. By Theorem~\ref{theorem: measurability} we have the following estimate,
	    \[	\|X^n_t\|\le \|X_0\|+\|V^n_t\|+\int_0^t e^{M(t-s)} \|f(s,S_s X_0+V^n_s)\| ds. \]
	    Hence,
	    \[	\sup\limits_{0\le t\le T} \|X^n_t\|^2\le 3 \|X_0\|^2+3\sup\limits_{0\le t\le T} \|V^n_t\|^2+3\sup\limits_{0\le t\le T} (\int_0^t e^{M(t-s)} \|f(s,S_s X_0+V^n_s)\| ds)^2, \]
	    where by Cauchy-Schwartz inequality we find
	    \[ \le 3 \|X_0\|^2+3\sup\limits_{0\le t\le T} \|V^n_t\|^2 + 3 T e^{2MT} \int_0^T \|f(s,S_s X_0+V^n_s)\|^2 ds, \]
	    now by Hypothesis~\ref{main_hypothesis}-(c) we have
	    \begin{multline*}
	    	\le 3\|X_0\|^2+3\sup\limits_{0\le t\le T} \|V^n_t\|^2 + 3 T e^{2MT} \int_0^T D(1+\|S_s X_0+V^n_s\|^2) ds \\
	    		\le 3\|X_0\|^2+3\sup\limits_{0\le t\le T} \|V^n_t\|^2 + 3 D T e^{2MT}\int_0^T (1+ 2\|X_0\|^2+2\|V^n_s\|^2) ds\\
	    		= 3DT^2e^{2MT} + (3+6DT^2e^{2MT}) \|X_0\|^2 + (3+6DT^2e^{2MT})\sup\limits_{0\le t\le T} \|V^n_t\|^2.
	    \end{multline*}
	    hence for completing the proof it suffices to show that $\mathbb{E}\sup\limits_{0\le t\le T} \|V^n_t\|^2<\infty$.

	    Applying Theorem~\ref{Theorem: Kotelenez inequality} we find,
	    \[ \mathbb{E}\sup\limits_{0\le t\le T} \|V^n_t\|^2 \le \mathbf{C} \mathbb{E} \left( \int_0^T \|g(s,X^{n-1}_{s})\|_{HS}^2 ds +  \int_0^T \int_E \|k(s,\xi,X^{n-1}_{s})\|^2 \nu(d\xi) ds \right)\]
	    where by Hypothesis~\ref{main_hypothesis}-(c),
	    \[ \le \mathbf{C} \mathbb{E} \int_0^T D(1+\|X^{n-1}_s\|^2) ds) \]
	    which is finite by induction Hypothesis.
		The basis of induction follows directly from Hypothesis~\ref{main_hypothesis}-(d).
	\end{proof}

    \begin{lemma}\label{lemma: proof of existence_supremum convergence}
    For $0<t\le T$ we have,
        \begin{equation} \label{equation: proof of existence_supremum convergence}
            \mathbb{E} \sup\limits_{0\le s\le t} \|X^{n+1}_s-X^n_s\|^2 \le C_0 C_1^n \frac{t^n}{n!}
        \end{equation}
				where $C_1= 2C(1+2\mathcal{C}_1^2) e^{4MT}$ and $C_0=\mathbb{E}\sup\limits_{0\le s\le T} \|X^1_s-X^0_s\|^2$.
				(Note that by Lemma~\ref{lemma: finite_second_moment_iteration}, $C_0<\infty$.)
    \end{lemma}

    \begin{proof}
        We prove by induction on $n$. The statement is obvious for $n=0$. Assume that the statement is proved for $n-1$. We have,
        \begin{equation}\label{equation: proof of existence_X^n+1-X^n}
            X^{n+1}_t-X^n_t= \int_0^t S_{t-s}(f(s,X^{n+1}_s)-f(s,X^n_s)) ds + \int_0^t S_{t-s} dM_s,
        \end{equation}
        where
        \begin{eqnarray*}
            M_t &=& \int_0^t (g(s,X^{n}_{s-})-g(s,X^{n-1}_{s-}))dW_s \\
            & & + \int_0^t \int_E {(k(s,\xi,X^{n}_{s-})-k(s,\xi,X^{n-1}_{s-}))\tilde{N}(ds,d\xi)}.
        \end{eqnarray*}
        Applying It\^o type inequality (Theorem~\ref{theorem:ito type inequality}), for $\alpha=0$, we have
        \begin{multline}\label{equation:proof of existence 1}
            \lVert X^{n+1}_t-X^n_t \rVert ^2 \le 2 \underbrace{\int_0^t {\langle X^{n+1}_{s-}-X^n_{s-},f(s,X^{n+1}_s)-f(s,X^n_s)\rangle ds}}_{A_t}\\
            + 2 \underbrace{\int_0^t {\langle X^{n+1}_{s-}-X^n_{s-},dM_s\rangle}}_{B_t} + [M]_t.
        \end{multline}
				 For the term $A_t$, the semimonotonicity assumption on $f$ implies
        \begin{equation}\label{equation:proof of existence_A_t}
            A_t \le M \int_0^t \|X^{n+1}_s-X^n_s\|^2 ds
        \end{equation}
        We also have
        \begin{multline*}
           \mathbb{E}[M]_t = \int_0^t \mathbb{E}\|g(s,X^{n}_s)-g(s,X^{n-1}_s)\|^2ds\\
           + \int_0^t \int_E \mathbb{E}\|k(s,\xi,X^n_s)-k(s,\xi,X^{n-1}_s)\|^2 \nu(d\xi) ds,
        \end{multline*}
        where by Hypothesis~\ref{main_hypothesis}-(b),
        \begin{equation}\label{equation:proof of existence_[M]_t}
            \le C \int_0^t \mathbb{E}\|X^n_s-X^{n-1}_s\|^2 ds.
        \end{equation}

        Applying Burkholder-Davies-Gundy inequality~(\cite{Peszat-Zabczyk},Theorem 3.50) , for $p=1$, to term $B_t$ we find,
        \begin{eqnarray*}
            \mathbb{E}\sup\limits_{0\le s\le t}|B_s| & \le & \mathcal{C}_1 \mathbb{E}\left([B]_t^\frac{1}{2}\right) \\
            & \le & \mathcal{C}_1 \mathbb{E}\left(\sup\limits_{0\le s\le t} (\|X^{n+1}_s-X^n_s\|) [M]_t^\frac{1}{2} \right)
        \end{eqnarray*}
        where $\mathcal{C}_1$ is the universal constant in the Burkholder-Davies-Gundy inequality. Applying Cauchy-Schwartz inequality we find,
        \begin{equation}\label{equation:proof of existence_B_t}
           \le \frac{1}{4} \mathbb{E} \sup\limits_{0\le s\le t} \|X^{n+1}_s-X^n_s\|^2 + \mathcal{C}_1^2 \mathbb{E}[M]_t.
        \end{equation}
        Now, taking supremum and then expectation on both sides of~\eqref{equation:proof of existence 1} and substituting~\eqref{equation:proof of existence_A_t},~\eqref{equation:proof of existence_[M]_t} and~\eqref{equation:proof of existence_B_t}, we find
        \begin{multline}\label{equation:proof of existence 3}
            \mathbb{E} \sup\limits_{0\le s\le t} \|X^{n+1}_s-X^n_s\|^2 \le 2M \int_0^t \mathbb{E}\|X^{n+1}_s-X^n_s\|^2 ds\\
            + C(1+2\mathcal{C}_1^2) \int_0^t \mathbb{E}\|X^{n}_s-X^{n-1}_s\|^2 ds\\
            + \frac{1}{2} \mathbb{E}(\sup\limits_{0\le s\le t} \|X^{n+1}_s-X^n_s\|)^{2}.
        \end{multline}

        The last term in the right hand side could be subtracted from the left hand side but for this subtraction to be valid it should be finite which is guaranteed by Lemma~\ref{lemma: finite_second_moment_iteration}. After subtraction we find,
            \[ \mathbb{E} \sup\limits_{0\le s\le t} \|X^{n+1}_s-X^n_s\|^2 \le 4 M \int_0^t \mathbb{E}\|X^{n+1}_s-X^n_s\|^2 ds + 2C(1+2\mathcal{C}_1^2) \int_0^t \mathbb{E}\|X^n_s-X^{n-1}_s\|^2 ds, \]
        Now let $h^n(t)=\mathbb{E}\sup\limits_{0\le s\le t} \|X^{n+1}_s-X^n_s\|^2$. Hence,
        \[ h^n(t)\le 4 M \int_0^t h^n(s) ds + 2C(1+2\mathcal{C}_1^2) \int_0^t h^{n-1}(s)ds \]

         Note that by Lemma~\ref{lemma: finite_second_moment_iteration}, $h^n(t)$ is bounded on $[0,T]$. Hence we can use Gronwall's inequality for $h^n(t)$ and find

\[ h^n(t) \le C_1 \int_0^t h^{n-1}(s) ds \]

where by induction hypothesis,

\[ \le C_1 \int_0^t C_0 C_1^{n-1} \frac{s^{n-1}}{(n-1)!} ds = C_0 C_1^n \frac{t^n}{n!} \]

which completes the proof.

\end{proof}

    Back to the proof of Theorem~\ref{theorem:existence and uniqueness}, since the right hand side of~\eqref{equation: proof of existence_supremum convergence} is a convergent series, $\{X^n\}$ is a cauchy sequence in $L^2(\Omega,\mathcal{F},\mathbb{P};L^\infty([0,T];H))$ and hence converges to a process $X_t(\omega)$. By choosing a subsequence they converge almost sure uniformly with respect to $t$, and since $\{X^n_t\}$ are adapted c\`adl\`ag, so is $X_t$.

    It remains to show that $X_t$ is a solution of~\eqref{mild_solution}.
    It suffices to show that the terms on both sides of equation~\eqref{equation: proof of existence_iteration} converge to that of~\eqref{mild_solution}. We know already that $X^n_t \to X_t$ in $L^2([0,T]\times\Omega;H)$. Moreover by Theorem~\ref{Theorem: Kotelenez inequality} we have,
    \begin{multline*}
        \mathbb{E}\|\int_0^t S_{t-s} g(s,X^n_{s-})dW_s - \int_0^t S_{t-s} g(s,X_{s-})dW_s\|^2\\
        \le \mathbf{C} \mathbb{E} \int_0^t\|g(s,X^n_s)-g(s,X_s)\|^2 ds\\
        \le \mathbf{C} C \int_0^t\mathbb{E} \|X^n_s-X_s\|ds \to 0,
    \end{multline*}
    and
    \begin{multline*}
        \mathbb{E}\|\int_0^t\int_E S_{t-s} k(s,\xi,X^n_{s-})d\tilde{N} - \int_0^t\int_E S_{t-s} k(s,\xi,X_{s-})d\tilde{N}\|^2\\
        \le \mathbf{C} \mathbb{E} \int_0^t\int_E\|k(s,\xi,X^n_s)-k(s,\xi,X_s)\|^2 \nu(d\xi)ds\\
        \le \mathbf{C} C \int_0^t\mathbb{E} \|X^n_s-X_s\|ds \to 0.
    \end{multline*}
    The term containing $f$ converges in the weak sense. Let $x\in H$,
    \begin{equation}\label{equation:proof of existence 5}
        \mathbb{E}\langle x,\int_0^t S_{t-s} (f(s,X^n_s)-f(s,X_s)) ds\rangle = \mathbb{E}\int_0^t \langle S_{t-s}^*x,f(s,X^n_s)-f(s,X_s)\rangle ds
    \end{equation}
    By demicontinuity of $f$, the integrand on the right hand side converges to $0$ for almost every $(s,\omega)\in [0,t]\times\Omega$. On the other hand, by Hypothesis~\ref{main_hypothesis}-(c), the integrand is dominated by a constant multiple of $\|x\| (1+\|X_s\|+\|X^n_s\|)$ where $\|X^n_s\| \to \|X_s\|$ pointwise almost everywhere and in $L^1([0,T]\times\Omega)$, hence by dominated convergence theorem we conclude that right hand side of~\eqref{equation:proof of existence 5} tends to $0$. Hence $X_t$ is a mild solution of~\eqref{main_equation}.
\end{proof}

\begin{remark}\label{remark: ordinary Wiener process}
    Although we have considered cylindrical Wiener processes but similar results hold for ordinary Wiener processes. In fact, let $W_t$ be a Wiener process on $K$, and $g(t,x,\omega):\mathbb{R}^+\times H\times \Omega \to L(K,H)$ satisfy Hypothesis~\ref{main_hypothesis} with the Hilbert-Schmidt norm replaced by the operator norm, then theorem~\ref{theorem:existence and uniqueness} and also theorems of the next sections hold also in this case. The reason is that if $\mathcal{K}$ is the RKHS of $W_t$ in the sense of~\cite{Peszat-Zabczyk}, then $W_t$ can be considered as a cylindrical Wiener process on $\mathcal{K}$ and the embedding $\mathcal{K}\hookrightarrow K$ is Hilbert-Schmidt, hence $g$ can be considered to take values in $L_{HS}(\mathcal{K},H)$ and satisfies Hypothesis~\ref{main_hypothesis}.
\end{remark}

\section{Continuity With Respect to Parameter}\label{section: Continuity With Respect to Parameter}

The continuous dependence of the solution of stochastic evolution equations with respect to initial conditions and coefficients has been studied by several authors. Consider the following stochastic evolution equation,
\begin{equation*}
    X_t=S_t X_0+\int_0^t S_{t-s}f(X_s) ds + V_t
\end{equation*}
Da Prato and Zabczyk~\cite{DaPrato_Zabczyk_paper} studied this equation in the case that $S_t$ is an analytic semigroup and $f$ is locally Lipschitz, and showed that the solution $X$ is a continuous function of $V$. Zangeneh~\cite{Zangeneh-Thesis} generalized this result and showed that the solution changes continuously when any or all of $V$, $f$,$A$ and $X_0$ are varied. Zangeneh~\cite{Zangeneh-Thesis} also generalized this result to stochastic evolution equations with Wiener noise and monotone nonlinearity. In the context of L\'evy noise, Albeverio, Mandrekar and R\"udiger~\cite{Albeverio-Mandrekar-Rudiger-2009} proved the continuous dependence of the solution of stochastic evolution equations with L\'evy noise and Lipshcitz coefficients.

In this section we show the continuous dependence of the mild solution of~\eqref{main_equation} on initial conditions and coefficients. The following theorem gives a bound for supremum distance of the mild solutions of two different equations by distances of their initial conditions and their coefficients.

\begin{theorem} [Continuity With Respect to Parameter I]\label{theorem: continuity I}
    Assume that for $n=0,1$, $f_n(t,x,\omega),g_n(t,x,\omega)$ and $k_n(t,\xi,x,\omega)$ satisfy Hypothesis~\ref{main_hypothesis} with the same constants. Let $X^n_t$ be the unique mild solution of
    \begin{multline*}
        dX^n_t= A X^n_t dt+f_n(t,X^n_t) dt + g_n(t,X^n_{t-})d W_t
        + \int_E k_n(t,\xi,X^n_{t-}) \tilde{N}(dt,d\xi),
    \end{multline*}
    with initial condition $X^n_0$. Then,
    \begin{multline}\label{equation: continuity}
        \mathbb{E} \sup\limits_{0\le t\le T} e^{-2\alpha t} \| X^1_t-X^0_t \| ^2 \le 2 e^{C_1T} \mathbb{E}\|X^1_0-X^0_0\|^2 \\
        + 2 e^{C_1 T} \int_0^T e^{-2\alpha t} \mathbb{E}\|f_1(t,X^0_t)-f_0(t,X^0_t)\|^2 dt\\
        + C_2 e^{C_1 T}\int_0^T e^{-2\alpha t} {\mathbb{E}\|(g_1(t,X^0_t)-g_0(t,X^0_t))\|^2 dt}\\
        + C_2 e^{C_1 T}\int_0^T \int_E{ e^{-2\alpha t} \mathbb{E}\|(k_1(t,\xi,X^0_t)-k_0(t,\xi,X^0_t))\|^2 \nu(d\xi) dt},
    \end{multline}
    for $C_1=4 M + 2 + C (8\mathcal{C}_1^2+4)$ and $C_2=8\mathcal{C}_1^2+4$ where $\mathcal{C}_1$ is the constant in Burkholder-Davies-Gundy inequality.
\end{theorem}
\begin{proof}
    First we consider the case that $\alpha=0$. Subtract $X^1$ and $X^0$:
    \begin{multline*}
        X^1_t-X^0_t=S_t (X^1_0-X^0_0)\\
        + \int_0^t S_{t-s} (f_1(s,X^1_s)-f_0(s,X^0_s))ds
        + \int_0^t S_{t-s} dM_s,
    \end{multline*}
    where
    \[ M_t=\int_0^t (g_1(s,X^1_{s-})-g_0(s,X^0_{s-}))dW_s+\int_E (k_1(s,\xi,X^1_{s-})-k_0(s,\xi,X^0_{s-}))d\tilde{N}. \]
    Applying It\^o type inequality (Theorem~\ref{theorem:ito type inequality}), for $\alpha=0$, to $X^1-X^0$ we find
    \begin{multline}\label{equation:proof of continuity 1}
        \| X^1_t-X^0_t \| ^2 \le \| X^1_0-X^0_0 \| ^2 +  2 \underbrace {\int_0^t {\langle X^1_{s-}-X^0_{s-} , (f_1(s,X^1_s)-f_0(s,X^0_s))\rangle ds}}_\bold{A_t}\\
        + 2 \underbrace {\int_0^t {\langle X^1_{s-}-X^0_{s-} , d M_s \rangle}}_\bold{B_t}+ [M]_t.
    \end{multline}
    We have
    \begin{multline*}
        \bold{A_t} =  \int_0^t {\langle X^1_{s-} - X^0_{s-} , f_1(s,X^1_s)-f_1(s,X^0_s) \rangle ds}\\
        + \int_0^t \langle X^1_{s-}-X^0_{s-} , f_1(s,X^0_s)-f_0(s,X^0_s) \rangle ds.
    \end{multline*}
    Using the monotonicity assumption and Cauchy-Schwartz inequality we have
    \begin{multline}\label{equation:proof of continuity 3}
        \bold{A_t} \le M \int_0^t \|X^1_{s}-X^0_{s}\|^2 ds + \frac{1}{2}\int_0^t \|X^1_{s}-X^0_{s}\|^2 ds\\
        + \frac{1}{2} \int_0^t \|f_1(s,X^0_{s})-f_0(s,X^0_{s})\|^2 ds.
    \end{multline}
    Applying Burkholder-Davies-Gundy inequality for $p=1$ to term $\mathbf{B_t}$ we find
        \[ \mathbb{E}\sup\limits_{0\le s\le t}\mathbf{|B_s|} \le \mathcal{C}_1\mathbb{E}\left(\sup\limits_{0\le s\le t} \|X^1_s-X^0_s\|[M]_t^\frac{1}{2}\right),\]
    and by Cauchy-Schwartz inequality,
    \begin{equation}\label{equation:proof of continuity 4}
        \le \frac{1}{4} \mathbb{E}\sup\limits_{0\le s\le t} \|X^1_s-X^0_s\|^2 + \mathcal{C}_1^2 \mathbb{E}[M]_t.
    \end{equation}
    We have
    \begin{eqnarray*}
        \mathbb{E}[M]_t &=& \int_0^t{\mathbb{E}\|(g_1(s,X^1_s)-g_0(s,X^0_s))\|^2 ds}\\
        &&+ \int_0^t\int_E{\mathbb{E}\|(k_1(s,\xi,X^1_s)-k_0(s,\xi,X^0_s))\|^2 \nu(d\xi) ds}\\
        & \le & 2 \int_0^t{\mathbb{E}\|(g_1(s,X^1_s)-g_1(s,X^0_s))\|^2 ds}\\
        && + 2\int_0^t{\mathbb{E}\|(g_1(s,X^0_s)-g_0(s,X^0_s))\|^2 ds} \\
        && + 2\int_0^t\int_E{\mathbb{E}\|(k_1(s,\xi,X^1_s)-k_1(s,\xi,X^0_s))\|^2 \nu(d\xi) ds}\\
        && +2\int_0^t\int_E{\mathbb{E}\|(k_1(s,\xi,X^0_s)-k_0(s,\xi,X^0_s))\|^2 \nu(d\xi) ds}.
    \end{eqnarray*}
    Using the Lipschitz assumption on $g$ and $k$ we find
    \begin{multline}\label{equation:proof of continuity 5}
        \mathbb{E}[M]_t \le 2C \int_0^t \mathbb{E}\|X^1_s-X^0_s\|^2 ds\\
        + 2\int_0^t{\mathbb{E} \|(g_1(s,X^0_s)-g_0(s,X^0_s))\|^2 ds}\\
        + 2\int_0^t\int_E{\mathbb{E} \|(k_1(s,\xi,X^0_s)-k_0(s,\xi,X^0_s))\|^2 \nu(d\xi) ds}.
    \end{multline}
    Substituting \eqref{equation:proof of continuity 3}, \eqref{equation:proof of continuity 4} and \eqref{equation:proof of continuity 5} in \eqref{equation:proof of continuity 1}, after cancellation we find
    \begin{eqnarray*}
        \mathbb{E} \sup\limits_{0\le s\le t} \| X^1_s-X^0_s \| ^2 &\le&  C_1 \int_0^t \mathbb{E}\|X^1_s-X^0_s\|^2 ds + 2 \mathbb{E}\|X^1_0-X^0_0\|^2  \\
        && + 2 \int_0^t \mathbb{E}\|f_1(s,X^0_s)-f_0(s,X^0_s)\|^2 ds\\
        && + C_2\int_0^t{\mathbb{E}\|(g_1(s,X^0_s)-g_0(s,X^0_s))\|^2 ds}\\
        && + C_2\int_0^t\int_E{\mathbb{E}\|(k_1(s,\xi,X^0_s)-k_0(s,\xi,X^0_s))\|^2 \nu(d\xi) ds},
    \end{eqnarray*}
    where $C_1=4 M + 2 + C(8\mathcal{C}_1^2+4)$ and $C_2=8\mathcal{C}_1^2+4$.

    Now applying Gronwall's inequality the statement follows. Hence the proof for the case $\alpha=0$ is complete. Now for the general case, apply the change of variables used in Lemma~\ref{lemma: alpha=0}.

\end{proof}

As a consequence of Theorem~\ref{theorem: continuity I} we prove that if the coefficients and initial conditions of a sequence of equations converge, then their mild solutions also converge to the mild solution of the limiting equation. The convergence that we prove is in a stronger sense than similar result in~\cite{Albeverio-Mandrekar-Rudiger-2009}.

\begin{corollary}[Continuity With Respect to Parameter II]\label{corollary: continuity II}
    Assume that for $n=0,1,2,\ldots$, $f_n$, $g_n$, $k_n$ and $X^n_0$ satisfy Hypothesis~\ref{main_hypothesis} with same constants and assume that for every $t \in [0,T]$ and $x\in H$ we have almost surely
    \begin{eqnarray*}
        &f_n(t,x,\omega)\to f_0(t,x,\omega)&\\
        &g_n(t,x,\omega)\to g_0(t,x,\omega)&\\
        &\int_E{\|k_n(t,\xi,x,\omega)-k_0(t,\xi,x,\omega)\|^2 \nu(d\xi)} \to 0&\\
        &\mathbb{E}\|X^n_0-X^0_0\|^2 \to 0.&
    \end{eqnarray*}
    Then
    \begin{equation*}
        \mathbb{E} \sup\limits_{0\le t\le T} \|X^n_t-X^0_t\|^2  \to 0.
    \end{equation*}
\end{corollary}

\begin{proof}
    Apply Theorem~\ref{theorem: continuity I} for $X^n$ and $X^0$. Note that by Hypothesis~\ref{main_hypothesis}-(c) the integrands on the right hand side of~\eqref{equation: continuity} are dominated by a constant multiple of $(1+\|X^0_t(\omega)\|^2)$, on the other hand by assumptions they tend to zero almost everywhere on $[0,T]\times\Omega$. Hence by dominated convergence theorem, the right hand side of~\eqref{equation: continuity} tends to $0$ and therefore
        \[ \mathbb{E} \sup\limits_{0\le t\le T} e^{-2\alpha t} \| X^1_t-X^0_t \| ^2 \to 0 \]
    which implies the statement.
\end{proof}

As another consequence of Theorem~\ref{theorem: continuity I} it follows that if the contraction coefficient of the semigroup is negative enough, then all the mild solutions are exponentially stable.

\begin{corollary} [Exponential Stability]\label{theorem: Exponential Stability}
    Let $X_t$ and $Y_t$ be mild solutions of~\eqref{main_equation} with initial conditions $X_0$ and $Y_0$. Then
    \begin{eqnarray*}
        \mathbb{E} \| X_t-Y_t \| ^2 &\le& 2 e^{\gamma t} \mathbb{E}\|X_0-Y_0\|^2
    \end{eqnarray*}
    for $\gamma= 2\alpha + 4 M + 2 + C(8\mathcal{C}_1^2+4)$. In particular, if $\gamma < 0$ then all mild solutions are exponentially stable.
\end{corollary}

\section{Markov Property}\label{section: Markov Property}

In this section we assume that $f$, $g$ and $k$ are deterministic functions and satisfy Hypothesis~\ref{main_hypothesis}. Let $0\le s \le t$ and $\eta: \Omega \to H$ be $\mathcal{F}_s$-measurable and square integrable. We denote by $X(s,\eta,t)$ the value at time $t$ of the solution of~\eqref{main_equation} starting at time $s$ from $\eta$. Let $B_b(H)$ be the space of real valued bounded measurable functions on $H$. For $\varphi \in B_b(H)$ and $x\in H$ define
\[ P_{s,t}\varphi (x):= \mathbb{E}\varphi(X(s,x,t)). \]
$P_{s,t}$ is called the \emph{transition semigroup}.

\begin{theorem}[Markov Property]\label{theorem: markov property}
    For $0\le r \le s \le t$ and $\varphi\in B_b(H)$ we have almost surely
    \[ \mathbb{E}\left( \varphi(X(r,x,t)|\mathcal{F}_s \right) = P_{s,t}\varphi (X(r,x,s)) \qquad \mathbb{P}-\mathrm{almost\,\,sure}. \]
\end{theorem}

\begin{proof}
    Let $C_b(H)$ denote the set of real valued bounded continuous functions on $H$. It suffices to prove the theorem for $\varphi\in C_b(H)$ since every $\varphi\in B_b(H)$ is the pointwise limit of a uniformly bounded sequence in $C_b(H)$. Fix $r$, $s$ and $t$. We claim that for any square integrable random variable $\eta(\omega)$ which is $\mathcal{F}_s$ measurable, we have
    \begin{equation}\label{equation: proof of Markov 1}
        \mathbb{E}\left( \varphi(X(s,\eta,t))|\mathcal{F}_s\right)=P_{s,t}\varphi(\eta(\omega)) \qquad \mathbb{P}-\mathrm{almost\,\,sure}.
    \end{equation}
    We first prove the claim for the case that $\eta$ has a simple form $\eta=\sum y_k \chi_{A_k}$, where $y_k\in H$ and $A_k\in\mathcal{F}_s$ form a partition of $\Omega$. We have
    \begin{eqnarray*}
        \mathbb{E}\left( \varphi(X(s,\eta,t))|\mathcal{F}_s\right)&=&\mathbb{E}\left( \sum \varphi(X(s,y_k,t)) \chi_{A_k}\big|\mathcal{F}_s\right)\\
        &=& \sum \chi_{A_k} \mathbb{E}\left( \varphi(X(s,y_k,t))|\mathcal{F}_s\right).
    \end{eqnarray*}
    Note that $X(s,y_k,t)$ is independent of $\mathcal{F}_s$, hence
    \begin{eqnarray*}
        &=& \sum \chi_{A_k} \mathbb{E}\left(\varphi(X(s,y_k,t))\right)\\
        &=& \sum \chi_{A_k} P_{s,t}\varphi(y_k)=P_{s,t}\varphi(\eta(\omega)).
    \end{eqnarray*}
    Now for general $\eta$ choose a sequence $\eta_n$ of simple random variables such that tend to $\eta$ in $L^2(\Omega)$ and almost surely. We then have
    \[ \mathbb{E}\left( \varphi(X(s,\eta_n,t))|\mathcal{F}_s\right)=P_{s,t}\varphi(\eta_n(\omega)) \qquad \mathbb{P}-\mathrm{almost\,\,sure}. \]
    Now let $n\to \infty$. By continuity with respect to initial conditions, the left hand side converges to $\mathbb{E}\left( \varphi(X(s,\eta,t))|\mathcal{F}_s\right)$ and the right hand side converges to $P_{s,t}\varphi(\eta(\omega))$ and~\eqref{equation: proof of Markov 1} follows. Now in~\eqref{equation: proof of Markov 1} let $\eta(\omega)=X(r,x,s)$. By uniqueness of solution we have $X(r,x,t)=X(s,X(r,x,s),t)$ and the theorem follows.
\end{proof}

\section{Some Examples}\label{section:examples}

In this section we provide some concrete examples of semilinear stochastic evolution equations with monotone nonlinearity and L\'evy noise. The examples consist of a stochastic delay differential equation and stochastic partial differential equations of parabolic and hyperbolic type.

\begin{example}[Stochastic Delay Equations]
    Consider the following delay differential equation in $\mathbb{R}$,
    \begin{equation}\label{equation: example_delay}
        \left\{\begin{array}{ll}
            dx(t)=&\left( \int_{-h}^0 \mu (d\theta)x(t+\theta) \right)dt + f(x(t))dt + g(x(t))dW_t + k(x(t)) dZ_t \\
            x(\theta)=& \psi(\theta), \theta \in (-h,0].
          \end{array} \right.
    \end{equation}
    where $h>0$, $\mu$ is a measure on $(-h,0]$ with finite variation, $W_t$ is a standard Wiener process in $\mathbb{R}$, $Z_t$ is a pure jump L\'evy martingale in $\mathbb{R}$ and $\psi(\theta)\in L^2((-h,0])$. Moreover assume that,

    \begin{hypothesis}\label{hypothesis: delay}
        \begin{description}

            \item[(a)] $f:\mathbb{R} \to \mathbb{R}$ is continuous and there exists a constant $M$ such that for any $a<b$,
                \[ f(a)-f(b) \le M (a-b),\]

            \item[(b)] $g:\mathbb{R} \to \mathbb{R}$ and $k:\mathbb{R} \to \mathbb{R}$ are Lipschitz.

            \item[(c)] There exists a constant $D$ such that for $a\in \mathbb{R}$,
                \[ | f(a)|^2 + | g(a)|^2 + |k(a)|^2 \le D (1+a^2).\]

        \end{description}
    \end{hypothesis}

    \begin{remark}
        Peszat and Zabczyk~\cite{Peszat-Zabczyk} have studied this delay differential equation with Lipschitz coefficients. We have replaced Lipschitzness of $f$ by the weaker assumption of semimonotonicity.
    \end{remark}

    Let $H=\mathbb{R}\times L^2((-h,0])$ and define the operator $A$ on $H$ by
        \[ A \left( {\begin{array}{c} u \\ v \end{array}} \right) = \left( {\begin{array}{c}  \int_{-h}^0 v(\theta) \mu(d\theta) \\ \frac{\partial v}{\partial \theta} \end{array}} \right). \]
    According to Da Prato and Zabczyk~\cite{DaPrato_Zabczyk_book}, Proposition A.25, the operator $A$ with domain
        \[ D(A)=\left\{ \left( {\begin{array}{c} u \\ v \end{array}} \right) \in H : v\in W^{1,2}(-h,0), v(0)=u \right\} \]
    generates a $C_0$ semigroup $S_t$ on $H$. Let $K=E=\mathbb{R}$ and let $\tilde{N}$ be the compensated Poisson random measure associated with $Z_t$. Define for $\left( {\begin{array}{c} u \\ v \end{array}} \right)\in H$ and $\xi\in \mathbb{R}$,
        \[ \bar{f}(u,v)= \left( {\begin{array}{c} f(u) \\ 0 \end{array}} \right), \bar{g}(u,v)= \left( {\begin{array}{c} g(u) \\ 0 \end{array}} \right), \bar{k}(\xi,u,v)= \left( {\begin{array}{c} \xi k(u) \\ 0 \end{array}} \right).\]
    It is easy to verify that $\bar{f}$, $\bar{g}$ and $\bar{k}$ satisfy Hypothesis~\ref{main_hypothesis}. Now, if we let
        \[X(t)= \left( {\begin{array}{c} x(t)\\ x_t \end{array}} \right) \]
    where $x_t(\theta)=x(t+\theta)$ for $\theta\in (-h,0]$, then equation~\eqref{equation: example_delay} can be written as
        \[dX(t)=A X(t) dt+ \bar{f}(X(t))dt+ \bar{g}(X(t^-))dW_t + \int_E \bar{k}(\xi, X(t^-)) \tilde{N}(dt,d\xi) \]
    with initial condition
        \[ X(0)= \left( {\begin{array}{c} \psi(0) \\ \psi \end{array}} \right) \]
    and hence by Theorem~\ref{theorem:existence and uniqueness} has a unique mild solution $x(t,\omega)$ with c\`adl\`ag trajectories and the solution depends continuously on initial condition.
\end{example}

\begin{example}[Stochastic Parabolic Equations with Finite Dimensional Noise] \label{example: finite_diemnsional_noise}
    In this example we consider a SPDE with L\'evy noise. The semimonotonicity assumption translates here to a simple assumption stated in Hypothesis~\ref{hypothesis: finite_dimensional}-(b) which includes, as special case, decreasing functions.

    Let $\mathcal{D}$ be a bounded domain with a smooth boundary in $\mathbb{R}^d$, and let $A$ be a self adjoint second order partial differential operator with smooth coefficients which is uniformly elliptic on $\mathcal{D}$. Let $B$ be the operator $B=d(x)\mathcal{D}_N+e(x)$, where $\mathcal{D}_N$ is the normal derivative on $\partial \mathcal{D}$ and $d$ and $e$ are in $C^\infty(\partial \mathcal{D})$.

    Consider the initial boundary value problem,
    \begin{equation}\label{equation: example_finite_dimentioanl_noise}
        \left\{\begin{array}{lrll}
            \frac{\partial u}{\partial t}  = Au + f(x,u(t,x))& + \sum\limits_{i=1}^m g_i(x,u(t^-,x)) \frac{\partial W_i}{\partial t}\\
            &+ \sum\limits_{j=1}^n  k_j(x,u(t^-,x)) \frac{\partial Z_j}{\partial t}& \textrm{on} & [0,\infty) \times \mathcal{D}\\
            Bu =0 && \textrm{on} & [0,\infty) \times \partial \mathcal{D}\\
            u(0,x)  =u_0(x) && \textrm{on} & \mathcal{D}.
        \end{array} \right.
    \end{equation}
    where $W_i(t), i=1,\ldots ,m$ are standard Wiener processes in $\mathbb{R}$ and $Z_j(t), j=1,\ldots ,n$ are pure jump L\'evy martingales in $\mathbb{R}$ with intensity measures $\nu_j(d\xi)$ and $u_0(x)\in L^2(\mathcal{D})$. Note that we have $\int_\mathbb{R} \xi^2\nu(d\xi) < \infty$. We assume moreover,
    \begin{hypothesis}\label{hypothesis: finite_dimensional}
        \begin{description}

            \item[(a)] $f(x,a):\mathcal{D} \times \mathbb{R} \to \mathbb{R}$, $g_i(x,a):\mathcal{D} \times \mathbb{R} \to \mathbb{R}$ and $k_j(x,a):\mathcal{D}\times \mathbb{R} \to \mathbb{R}$ satisfy Caratheodory condition, i.e. they are continuous with respect to $u$ for almost all $x\in \mathcal{D}$ and are measurable with respect to $x$ for all values of $u$.

            \item[(b)] There exists a real constant $M$ such that for any $x\in \mathcal{D}$ and real numbers $a < b$,
                \[ f(x,a)-f(x,b) \le M (a-b).\]

            \item[(c)] There exists a constant $C>0$ such that for any $x\in \mathcal{D}$ and $a,b\in\mathbb{R}$,
                \[ \sum\limits_{i=1}^m |g_i(x,a)-g_i(x,b)|^2 + \sum\limits_{j=1}^n |k_j(x,a)-k_j(x,b)|^2 \le C |a-b|^2.\]

            \item[(d)] There exists a function $\psi(x) \in L^2(\mathcal{D})$ and a constant $D>0$ such that for any $x\in \mathcal{D}$ and $a\in\mathbb{R}$,
                \[ |f(x,a)| + \sum\limits_{i=1}^m |g_i(x,a)| + \sum\limits_{j=1}^n |k_j(x,a)| \le \psi(x) + D |a|.\]
        \end{description}
    \end{hypothesis}

    Let $H=L^2(\mathcal{D})$. The operator $A$ together with the boundary condition generates a contraction semigroup $S_t$ on $H$ (Peszat and Zabczyk~\cite{Peszat-Zabczyk} section B.2). Define for $u(x)\in L^2(\mathcal{D})$,
        \[ \bar{f}(u)(x)=f(x,u(x))\]
        \[ \bar{g}_i(u)(x)=g_i(x,u(x))\]
        \[ \bar{k}_j(u)(x)=k_j(x,u(x))\]

    Since $f$, $g_i$ and $k_j$ satisfy Hypothesis~\ref{hypothesis: finite_dimensional}-(a) and~\ref{hypothesis: finite_dimensional}-(d), then by Theorem (2.1) of Krasnosel'ski\u\i~\cite{Krasnoelskii}, $\bar{f}$, $\bar{g}_i$ and $\bar{k}_j$ define continuous operators from $L^2(\mathcal{D})$ to $L^2(\mathcal{D})$ and for a suitable constant $D^\prime$ satisfy
        \[ \|\bar{f}(u)\| + \|\bar{g}_i(u)\| + \|\bar{k}_j(u)\| \le D^\prime (1+\|u\|) \]

    Define $\bar{g}:L^2(\mathcal{D})\to L^2(\mathcal{D})^m$ and $\bar{k}:\mathbb{R}^n \times L^2(\mathcal{D})\to L^2(\mathcal{D})^n$ by
        \[ \bar{g}=(\bar{g}_1,\ldots,\bar{g}_m) \]
        \[ \bar{k}(\xi,u)=(\xi_1 \bar{k}_1(u),\ldots,\xi_n \bar{k}_n(u)). \]
    Now, it is straightforward to verify that $\bar{f}$, $\bar{g}$ and $\bar{k}$ satisfy Hypothesis~\ref{main_hypothesis}.

    Let $W(t)=(W_1(t),\ldots,W_m(t))$ be a Wiener process on $K=\mathbb{R}^m$ and let $E=\mathbb{R}^n$ and $\tilde{N}(dt,d\xi)$ be the compensated Poisson random measure associated with the L\'evy process $Z(t)=(Z_1(t),\ldots,Z_n(t))$. Now we can write~\eqref{equation: example_finite_dimentioanl_noise} in the form of equation~\eqref{main_equation},
        \[ du(t) =Au(t) dt + \bar{f}(u(t)) dt + \bar{g}(u(t^-)) dW_t + \int_E \bar{k}(\xi,u(t^-)) \tilde{N}(dt,d\xi) \]
    with initial condition $u_0$, and hence equation~\eqref{equation: example_finite_dimentioanl_noise} has a unique mild solution $u(t,x,\omega)$ with values in $L^2(\mathcal{D})$ and with c\`adl\`ag trajectories. The solution also depends continuously on initial condition.
\end{example}

\begin{example}[Stochastic Parabolic Equations with Space-Time Noise] \label{example:general_parabolic}
    In this example we would like to consider a SPDE with infinite dimensional noise. A natural candidate for infinite dimensional noise is space-time white noise, but it can be shown that in dimensions greater than one, even the equation
        \[ \frac{\partial u}{\partial t}(t,x) = \Delta u(t,x) + \dot{W} (t,x) \]
    does not have a function valued solution (\cite{Peszat-Zabczyk}, Remark 12.2). In order to guarantee the existence of solution we assume that coefficients are operators on certain function spaces.

    Let $\mathcal{D}$ and $A$ be as in Example~\ref{example: finite_diemnsional_noise}. Consider the initial boundary value problem
    \begin{equation}\label{equation: example_parabolic}
        \left\{\begin{array}{lrll}
             \frac{\partial u}{\partial t} =& Au + f(u(t)) + g(u(t^-)) \frac{\partial W}{\partial t} & & \\
             &+ k(u(t^-)) \frac{\partial Z}{\partial t} & \textrm{on} & [0,\infty) \times \mathcal{D}  \\
             u=0 && \textrm{on} &  [0,\infty) \times \partial \mathcal{D} \\
             u(0,x)=0 && \textrm{on} & \mathcal{D}
          \end{array} \right.
    \end{equation}
    where $W_t$ is a cylindrical Wiener process on $L^2(\mathcal{D})$ and $Z_t$ is a pure jump L\'evy martingale on $L^2(\mathcal{D})$, and by $u(t)$ we mean $u(t,.)$.

    Let $n$ be an integer. We wish to solve this equation in the function space $H_n$ introduced in Walsh~\cite{Walsh}. Let $\{\phi_j\}$ be the complete orthonormal basis for $L^2(\mathcal{D})$ consisting of eigenfunctions of $A$ with Dirichlet boundary condition and $-\lambda_j<0$ be the corresponding eigenvalues. Let $H_n$ be the Hilbert space that has as a complete orthonormal basis the set $\{e_j=(1+\lambda_j)^{-\frac{n}{2}} \phi_j\}$. Obviously $H_0=L^2(\mathcal{D})$ and the spaces $H_n$ can be continuously embedded in each other as
        \[ \cdots \subset H_n \subset \cdots \subset H_1 \subset L^2(\mathcal{D}) \subset H_{-1} \subset \cdots \subset H_{-n} \subset \cdots. \]

    Assume moreover,

    \begin{hypothesis}\label{hypothesis: space-time noise parabolic}
        \begin{description}

            \item[(a)] $f:H_n \to H_n$ is measurable, demicontinuous and there exists a constant $M$ such that for any $u,v \in H_n$,
                \[ \langle f(u)-f(v),u-v \rangle \le M \|u-v\|^2,\]

            \item[(b)] $g:H_n \to L_{HS}(L^2(\mathcal{D}),H_n)$ and $k:H_n \to L(L^2(\mathcal{D}),H_n)$ are Lipschitz.

            \item[(c)] There exists a constant $D$ such that for $u\in H_n$,
                \[ \| f(u)\|^2 + \| g(u)\|^2 + \|k(u)\|^2 \le D (1+\|u\|^2),\]

        \end{description}
    \end{hypothesis}

    $A$ generates a $C_0$ semigroup $S_t$ on $H$ where $S_t e_j = e^{-t\lambda_j} e_j$. Let $K=E=L^2(\mathcal{D})$ and let $\tilde{N}(dt,d\xi)$ be the compensated Poisson random measure on $E$ corresponding to the L\'evy process $Z_t$ with intensity measure $\nu(d\xi)$, and define
        \[ \bar{k}(\xi,u):= k(u)(\xi) \]
    Now, it is easy to verify that $f$, $g$ and $\bar{k}$ satisfy Hypothesis~\ref{main_hypothesis} and therefore equation~\eqref{equation: example_parabolic} can be written in the form of equation~\eqref{main_equation} with initial condition $0$ and hence~\eqref{equation: example_parabolic} has a mild solution $u(t,x,\omega)$ with values in $H_n$ and with c\`adl\`ag trajectories.

    \begin{remark}
        In Hypothesis~\ref{hypothesis: space-time noise parabolic}-(b) one can replace the condition on $g$ by
            \[ g:H_n \to L(W^{-p,2}(\mathcal{D}),H_n) \]
        where $p>\frac{d}{2}$ is a real number, since the embedding $L^2(\mathcal{D})\hookrightarrow W^{-p,2}(\mathcal{D})$ is Hilbert-Schmidt (see Walsh~\cite{Walsh} page 334).
    \end{remark}
\end{example}

\begin{example}[Second Order Stochastic Hyperbolic Equations] \label{example: second_order_hyperbolic}
    In this example we consider a second order hyperbolic SPDE with L\'evy noise. The semimonotonicity condition on the drift coefficient translates here to being semimonotone with respect to the second variable and being Lipschitz with respect to the first variable. Let $\mathcal{D}$, $A$, $W_t$ and $Z_t$ be as in Example~\ref{example:general_parabolic}.

    Consider the following second order equation in $\mathcal{D}$:
    \begin{equation}\label{equation: example_hyperbolic}
        \left\{\begin{array}{lrll}
             \frac{\partial^2}{\partial t^2} u(t,x) = & Au + f(u(t),\frac{\partial u}{\partial t}) + g(u(t^-)) \frac{\partial W}{\partial t} & & \\
             & + k(u(t^-)) \frac{\partial Z}{\partial t}  & \textrm{on} &  [0,\infty) \times \mathcal{D} \\
             u=0 && \textrm{on} &  [0,\infty) \times \partial \mathcal{D} \\
             u(0,x)=0 && \textrm{on} & \mathcal{D} \\
            \frac{\partial u}{\partial t} (0,x)=0 && \textrm{on} & \mathcal{D}.
          \end{array} \right.
    \end{equation}

    Let $n$ be an integer. We wish to solve this equation in the function space $H_{n+1}$ introduced in Example~\ref{example:general_parabolic}. Assume that,

    \begin{hypothesis}\label{hypothesis: Second Order Hyperbolic}
        \begin{description}

            \item[(a)] $f:H_{n+1}\times H_n \to H_n$ is measurable, demicontinuous and there exists constants $M$ and $C$ such that for any $u,u_1,u_2 \in H_{n+1}, v,v_1,v_2\in H_n$,
                \[ \langle f(u,v_1)-f(u,v_2),v_1-v_2 \rangle \le M \|v_1-v_2\|^2, \]
                \[ \|f(u_1,v)-f(u_2,v)\|\le C \|u_1-u_2\|. \]

            \item[(b)] $g:H_{n+1} \to L_{HS}(L^2(\mathcal{D}), H_n)$ and $k:H_{n+1} \to L(L^2(\mathcal{D}), H_n)$ are Lipschitz.

            \item[(c)] There exists a constant $D$ such that for $u \in H_{n+1}$, and $v \in H_n$
                \[ \| f(u,v)\|^2 + \| g(u)\|^2 + \|k(u)\|^2 \le D (1+\|u\|^2+\|v\|^2).\]

        \end{description}
    \end{hypothesis}

    Let $H=H_{n+1}\times H_n$. Note that $A$ is self adjoint and negative definite on $H_n$. Moreover, we have
        \[ D((-A)^\frac{1}{2})=  H_{n+1}. \]
    Hence by Lemma B.3 of~\cite{Peszat-Zabczyk}, the operator
        \[ \mathcal{A}=\left( {\begin{array}{cc} 0&I\\ A&0 \end{array}} \right)\]
    generates a $C_0$ semigroup of contractions on $H$. Let $K=E=L^2(\mathcal{D})$. We also define
        \[ \bar{f}(u,v)= \left( {\begin{array}{c} 0\\ f(u,v) \end{array}} \right), \bar{g}(u,v)(\phi)= \left( {\begin{array}{c} 0\\ g(u)(\phi) \end{array}} \right), \bar{k}(\xi,u,v)= \left( {\begin{array}{c} 0\\ k(u)(\xi) \end{array}} \right) \]
    We claim that $\bar{f}$, $\bar{g}$ and $\bar{k}$ satisfy Hypothesis~\ref{main_hypothesis}. We show the semimonotonicity condition, the other conditions are straightforward.
        \begin{eqnarray*}
            \langle \bar{f}(u_1,v_1) - \bar{f}(u_2,v_2) , \left( {\begin{array}{c} u_1\\ v_1 \end{array}} \right) - \left( {\begin{array}{c} u_2\\ v_2 \end{array}} \right) \rangle = \langle \bar{f}(u_1,v_1) - \bar{f}(u_2,v_2) , v_1-v_2 \rangle &&\\
            = \langle \bar{f}(u_1,v_1) - \bar{f}(u_1,v_2) , v_1-v_2 \rangle + \langle \bar{f}(u_1,v_2) - \bar{f}(u_2,v_2) , v_1-v_2 \rangle &&
        \end{eqnarray*}
    where by Hypothesis~\ref{hypothesis: Second Order Hyperbolic}-(a) and Shwartz inequality
        \begin{eqnarray*}
            \le M \|v_1-v_2\|^2 + C \|u_1-u_2\| \|v_1-v_2\| \le (M+C)\left( \|u_1-u_2\|^2 + \|v_1-v_2\|^2 \right)
        \end{eqnarray*}
    Hence Hypothesis~\ref{main_hypothesis}-(a) holds with constant $M+C$. Now, if we let
        \[X(t)= \left( {\begin{array}{c} u(t)\\ \frac{\partial u}{\partial t}(t) \end{array}} \right) \]
    then equation~\eqref{equation: example_hyperbolic} can be written as
        \[dX(t)=\mathcal{A}X(t) dt+ \bar{f}(X(t))dt+ \bar{g}(X(t^-))dW_t + \int_E \bar{k}(\xi,X(t^-)) \tilde{N}(dt,d\xi) \]
    with initial condition $0$ and hence by Theorem~\ref{theorem:existence and uniqueness} has a mild solution $u(t,x,\omega)$ with values in $H_{n+1}$ and with c\`adl\`ag trajectories.

\begin{remark}
    One can generalize equation~\eqref{equation: example_hyperbolic} by assuming that coefficients $g$ and $k$ depend moreover on $\frac{\partial u}{\partial t}$. It suffices to modify the domain of $g$ and $k$ to $H_{n+1}\times H_n$ and then the same arguments hold.
\end{remark}

\end{example}

\end{document}